\documentclass[12pt]{amsart}
\usepackage{amssymb,amsmath,amsthm}
\numberwithin{equation}{section}
\begin{document}

\theoremstyle{plain}
\newtheorem{theorem}{Theorem}[section]
\newtheorem{lemma}[theorem]{Lemma}
\newtheorem{proposition}[theorem]{Proposition}
\newtheorem{corollary}[theorem]{Corollary}
\newtheorem{conjecture}[theorem]{Conjecture}

\theoremstyle{definition}
\newtheorem*{definition}{Definition}

\theoremstyle{remark}
\newtheorem*{remark}{Remark}
\newtheorem{example}{Example}[section]
\newtheorem*{remarks}{Remarks}

\newcommand{\cc}{{\mathbf C}}
\newcommand{\qq}{{\mathbf Q}}
\newcommand{\rr}{{\mathbf R}}
\newcommand{\nn}{{\mathbf N}}
\newcommand{\zz}{{\mathbf Z}}
\newcommand{\pp}{{\mathbf P}}
\newcommand{\al}{\alpha}
\newcommand{\be}{\beta}
\newcommand{\ga}{\gamma}
\newcommand{\ze}{\zeta}
\newcommand{\om}{\omega}
\newcommand{\ep}{\epsilon}
\newcommand{\la}{\lambda}
\newcommand{\de}{\delta}
\newcommand{\De}{\Delta}
\newcommand{\Ga}{\Gamma}
\newcommand{\si}{\sigma}

\title{On the length of binary forms}

\author{Bruce Reznick}
\address{Department of Mathematics, University of 
Illinois at Urbana-Champaign, Urbana, IL 61801} 
\email{reznick@math.uiuc.edu}

\subjclass[2000]{Primary: 11E76, 11P05, 14N10}
\begin{abstract}
The $K$-length of a form $f$ in $K[x_1,\dots,x_n]$, $K \subset \cc$, is
the smallest number of $d$-th powers of linear forms of which $f$ is a
$K$-linear combination. We present many results, old and new, about
$K$-length, mainly in $n=2$, and often about the length of the same
form over different fields. For example, the $K$-length of $3x^5
-20x^3y^2+10xy^4$ is three for $K = \qq(\sqrt{-1})$,  four for $K =
\qq(\sqrt{-2})$ and five for $K = \rr$.
\end{abstract}
\date{\today}
\maketitle

\maketitle

\section{Introduction and Overview}

Suppose $f(x_1,...,x_n)$ is a form of degree $d$ with coefficients in
a  field $K \subseteq \cc$. The {\it
$K$-length of $f$}, $L_K(f)$, is the smallest $r$  for which there is a
representation 
\begin{equation}\label{E:basic}
f(x_1,\dots,x_n) = \sum_{j=1}^r \la_j\bigl(\al_{j1}x_1 + \dots +
\al_{jn}x_n\bigr)^d
\end{equation}
with $ \la_j, \al_{jk} \in K$.

In this paper, we consider the $K$-length of a
fixed form $f$ as $K$ varies; this is apparently an open question
in the literature, even for binary forms ($n=2$). Sylvester \cite{S1,S2}
explained how to compute $L_{\cc}(f)$ for binary forms in 1851. Except
for a few remarks, we  shall restrict our attention to binary forms.

It is trivially true that $L_K(f) = 1$ for linear $f$ and for $d=2$, 
$L_K(f)$ equals the rank of $f$: a representation over $K$ can be
found by completing the 
square, and this length cannot be shortened by enlarging the
field. Accordingly, we shall also assume that $d \ge 3$. 
Many of our results  are extremely low-hanging fruit which were either 
known in the 19th century, or would have been, had its mathematicians 
been able to take  21st century  undergraduate mathematics courses.

When $K = \cc$, the $\la_j$'s in \eqref{E:basic}  are
superfluous. The computation of $L_{\cc}(f)$ is a huge, venerable and
active subject, and very hard when $n \ge 3$. The
interested reader is directed to \cite{C, CC, CM, DK, ER,Ge, He, IK,
  LT, RS, Re1, Re2} as representative recent works. Even for small
$n,d \ge 3$, there are still many open questions.  Landsberg and
Teitler \cite{LT} complete a classification of $L_{\cc}(f)$ for
ternary cubics $f$ and also discuss $L_{\cc}(x_1x_2\cdots x_n)$, among
other topics. Historically, much attention
has centered on the $\cc$-length of a {\it general} form of degree
$d$. In 1995, Alexander and Hirschowitz \cite{AH} (see also
\cite{BO,Mi}) established that for 
$n,d \ge 3$, this length is $\lceil \frac 1n
\binom{n+d-1}{n-1}\rceil$, the constant-counting value,
with the four exceptions known since the 19th century -- $(n,d) = (3,5),(4,3), 
(4,4),(4,5)$ -- in which the length is $\lceil \frac 1n
\binom{n+d-1}{n-1}\rceil + 1$.

An alternative definition would remove the coefficients from
\eqref{E:basic}. (The computation of the alternative definition is
likely to be much harder than the one we consider here, if for 
no other reason than that cones are harder to work with than
subspaces.)   
This alternative definition was considered by Ellison \cite{E1} in the
special cases $K = \cc, \rr, \qq$. When $d$ is odd and $K= \rr$, the
$\la_j$'s are also unnecessary.  
When $d$ is even and $K \subseteq
\rr$, it is not easy to determine whether \eqref{E:basic} is possible
for a given $f$. In \cite{R1}, the principal object of study is
$Q_{n,2k}$, the (closed convex) cone of forms which are a sum ($\la_j =
1$) of $2k$-th powers of real linear forms. 
 As two illustrations of the difficulties which can arise in this case:
 $\sqrt 2$ is not totally
positive in $K = \qq(\sqrt 2)$, so $\sqrt 2\ x^2$ is not a sum of squares
in $K[x]$, and $x^4 + 6\la x^2y^2 + y^4 \in Q_{2,4}$ if and only if
$\la \in [0,1]$.  
For more on the possible signs that may arise in a minimal $\rr$-representation
\eqref{E:basic}, see \cite{R4}. 
Helmke \cite{He} uses both definitions for  length for forms, and is mainly
concerned with the coefficient-free version in the case when $K$ is an 
algebraically closed (or a real closed) field of characteristic zero,
not necessarily a subset of $\cc$.  
 Newman and Slater \cite{NS} do not restrict
to homogeneous polynomials. They write $x$ as a sum of $d$ $d$-th
powers of linear polynomials; by substitution, any polynomial
is a sum of at most $d$ $d$-th powers of polynomials. They also show
that the minimum 
number of $d$-th powers in this formulation is $\ge \sqrt d$. Because
of the degrees of the summands, these methods do not 
homogenize to forms. 
Mordell \cite{Mo} showed that a polynomial that is a
sum of cubes of linear forms over $\zz$ is also a sum of at most eight
such cubes. More generally, if $R$ is a commutative 
ring, then its {\it $d$-Pythagoras number}, $P_d(R)$, is the
smallest integer $k$ so that any sum of $d$-th powers in $R$ is a sum
of $k$ $d$-th powers.
This subject is closely related to Hilbert's 17th Problem; see
\cite{CDLR, CLR, CLPR}.

Two examples illustrate the phenomenon of multiple lengths over
different fields.

\begin{example}
Suppose  $f(x,y) = (x + \sqrt 2 y)^d + (x - \sqrt 2 y)^d \in \qq[x,y]$.
Then  $L_K(f)$ is 2 (if $\sqrt 2 \in K$) and $d$ (otherwise). This
example first appeared in  \cite[p.137]{R1}. (See Theorem 4.6 for a
generalization.) 
\end{example}

\begin{example} If $\phi(x,y) = 3x^5 -20x^3y^2+10xy^4$,
then $L_K(\phi) = 3$ if and only if $\sqrt{-1} \in K$,  
$L_K(\phi)= 4$ for
$K=\qq(\sqrt{-2}),\qq(\sqrt{-3}),\qq(\sqrt{-5}),\qq(\sqrt{-6})$ (at least)
and $L_{\rr}(\phi) = 5$. (We give proofs of these assertions
 in Examples 2.1 and 3.1.)
\end{example}

The following simple definitions and remarks apply in the
obvious way to forms in $n \ge 3$ variables, but for simplicity are
given for binary forms.
A representation such as \eqref{E:basic} is called {\it $K$-minimal} if $r
= L_K(f)$.  
Two linear forms are called {\it  distinct} if they (or their $d$-th powers)
are not proportional. A  representation is {\it   honest} if the
summands are pairwise distinct. Any minimal representation is honest.
Two honest representations are {\it 
  different} if the ordered sets of summands are not rearrangements of
each other;  we do not distinguish between
$\ell^d$ and $(\zeta_d^k \ell)^d$ where $\zeta_d
= e^{2\pi i/d}$.

If $g$ is obtained from $f$ by an invertible linear change of
variables over $K$, then $L_K(f) = L_K(g)$.
Given a form $f \in \cc[x,y]$, the field generated by the
coefficients of $f$ over $\cc$ is denoted $E_f$.  The $K$-length can
only be defined for fields $K$ satisfying $E_f \subseteq K \subseteq \cc$. 
The following implication is immediate:
\begin{equation}\label{E:enlarge}
 K_1 \subset K_2 \implies L_{K_1}(f) \ge L_{K_2}(f).
\end{equation}
 Strict inequality in \eqref{E:enlarge} is possible, as shown by the two
examples. The {\it cabinet} of $f$,
$\mathcal C(f)$ is the set of all possible lengths for $f$.

We now outline the remainder of the paper. 

In section two, we give a self-contained proof of Sylvester's 1851
Theorem  (Theorem 2.1). Although originally given over $\cc$, it
adapts easily to any 
 $K \subset \cc$ (Corollary 2.2). If $f$ is a binary form, then
 $L_K(f) \le r$ iff a certain 
subspace of the binary forms of degree $r$ (a subspace determined by   $f$)   
contains a form that splits into distinct factors over $K$. We
illustrate the algorithm by proving the 
assertions of lengths 3 and 4 for $\phi$ in Example 1.2.

In section three, we prove (Theorem 3.2) a homogenized
version of Sylvester's 1864 Theorem (Theorem 3.1), which implies that
if real $f$ has $r$ linear factors  over $\rr$ (counting
multiplicity), then $L_{\rr}(f) \ge r$. In particular, 
$L_{\rr}(\phi) = 5$. As far as we have been able to tell, Sylvester
did not connect his two theorems: perhaps because he presented the
second one for non-homogeneous polynomials in a single variable, perhaps
because ``fields'' had not yet been invented.   

We apply these theorems and some other simple observations in sections
four and five.  
 We first show that if
$L_{\cc}(f) = 1$, then $L_{E_f}(f) = 1$ as well (Theorem 4.1).
Any set  of $d+1$ $d$-th powers of pairwise distinct linear forms is
linearly independent (Theorem 4.2).  
It follows quickly that if $f(x,y)$ has two different honest
representations of length $r$ and $s$, then $r+s\ge d+2$ (Corollary
4.3), and so if
$L_{E_f}(f) = r \le \frac{d+1}2$, then the representation over $E_f$
is the unique minimal $\cc$-representation (Corollary 4.4). We show
that Example 1.1 gives the template for forms $f$ satisfying $L_{\cc}(f) = 2 <
L_{E_f}(f)$ (see Theorem 4.6), and give two generalizations which
provide other types of constructions (Corollaries 4.7 and 4.8) of
forms with multiple lengths.  
We apply Sylvester's 1851 Theorem to  give an easy proof of the known
result that $L_{\cc}(f) \le d$ (Theorem 4.9) and a slightly  
trickier proof of the probably-known result that $L_K(f) \le d$ as well (Theorem
4.10). Theorem 4.10 combines with Theorem 3.2 into Corollary 4.11:
if $f \in \rr[x,y]$ is a product of $d$ linear factors,
then $L_{\rr}(f) = d$. Conjecture 4.12 asserts that $f \in \rr[x,y]$
is a product of $d$ linear factors if and only if $L_{\rr}(f) = d$.

 In Corollary 5.1, we discuss the
various possible cabinets when $d=3,4$; and give examples for each one
not already forbidden.
We then completely classify binary cubics; the key point 
of Theorem 5.2 is that if the cubic $f$ has no repeated factors, then $L_k(f) =
2$ if and only  if $E_f(\sqrt{-3\Delta(f)}) \subseteq K$; this
significance of the  discriminant $\Delta(f)$ can already be found
e.g. in Salmon  \cite[\S 167]{Sa}. This proves Conjecture 4.12 for
$d=3$. In Theorem 5.3, we show that Conjecture 4.12 also holds for $d=4$.  
 Another probably old theorem (Theorem 5.4) is that $L_{\cc}(f) = d$
if and only if there are distinct linear forms $\ell, \ell'$ so that
$f = \ell^{d-1}\ell'$. The minimal representations of
$x^ky^k$ are parameterized (Theorem 5.5), and in Corollary 5.6, we
show that $L_{K}((x^2+y^2)^k) \ge k+1$, with equality if and only if
$\tan\frac {\pi}{k+1} \in K$. In particular, $L_{\qq}((x^2+y^2)^2) =
4$. Theorem 5.7  shows that $L_{\qq}(x^4 + 6\la x^2y^2 + y^4) = 3$ if
and only if a certain quartic diophantine equation over $\zz$
has a non-zero solution.

Section six lists some open questions.

We would like to express our appreciation to the organizers of the
{\it Higher Degree Forms} conference in Gainesville in May 2009 for
offering the opportunities to speak on these topics, and to write this
article for its Proceedings. We also thank Mike Bennett, Joe Rotman
and Zach Teitler for helpful conversations. 

\section{Sylvester's 1851 Theorem}

Modern proofs of Theorem 2.1 can be found in the work of Kung and Rota:
\cite[\S 5]{KR}, with further discussion in \cite{K1,K2,K3,R3}.
We present here a very elementary  proof showing the connection with constant
coefficient linear recurrences, in the hopes that this remarkable
theorem might become better known to the modern reader.

\begin{theorem}[Sylvester]
Suppose 
\begin{equation}\label{E:S1}
f(x,y) = \sum_{j=0}^d \binom dj a_j x^{d-j}y^j
\end{equation}
 and suppose 
\begin{equation}\label{E:S2}
h(x,y)
= \sum_{t=0}^r 
c_tx^{r-t}y^t = \prod_{j=1}^{r} ( - \beta_j x + \alpha_j y)
\end{equation}
 is  a product of pairwise distinct linear factors.   Then there exist
 $\lambda_k\in \cc$ so that  
\begin{equation}\label{E:S3}
f(x,y) = \sum_{k=1}^r \lambda_k (\alpha_k x + \beta_k y)^d
\end{equation}
if and only if
\begin{equation}\label{E:S4}
\begin{pmatrix}
a_0 & a_1 & \cdots & a_r \\
a_1 & a_2 & \cdots & a_{r+1}\\
\vdots & \vdots & \ddots & \vdots \\
a_{d-r}& a_{d-r+1} & \cdots & a_d
\end{pmatrix}
\cdot
\begin{pmatrix}
c_0\\c_1\\ \vdots \\ c_r
\end{pmatrix}
=\begin{pmatrix}
0\\0\\ \vdots \\ 0
\end{pmatrix};
\end{equation}
that is, if and only if 
\begin{equation}\label{E:S5}
\sum\limits_{t=0}^r a_{\ell + t} c_t = 0, \qquad \ell = 0,1,\dots,  d-r.
\end{equation}
\end{theorem}

\begin{proof}
First suppose that \eqref{E:S3} holds. Then for $0 \le j \le d$,  
\begin{equation*}
\begin{gathered}
a_j = \sum_{k=1}^r \la_k \al_k^{d-j}\be_k^j \implies 
\sum_{t=0}^r a_{\ell+t} c_t =  \sum_{k=1}^r
\sum_{t=0}^r\la_k \al_k^{d-\ell-t}\be_k^{\ell+t} c_t  \\ =
\sum_{k=1}^r \la_k \al_k^{d-\ell-r}\be_k^{\ell} \sum_{t=0}^r 
\al_k^{r-t}\be_k^t c_t =  \sum_{k=1}^r \la_k \al_k^{d-\ell-r}\be_k^{\ell}
\ h(\al_k,\be_k) = 0. 
\end{gathered}
\end{equation*}

Now suppose that \eqref{E:S4} holds and suppose first that $c_r \neq
0$. We may assume without loss of generality that $c_r
= 1$ and that $\al_j=1$ in \eqref{E:S2}, so that the $\be_j$'s are distinct.
Define the {\it infinite}
sequence $(\tilde a_j)$, $j \ge 0$, by:
\begin{equation}\label{E:inf}
 \tilde a_j = a_j\quad \text{if} \quad 0 \le j \le r-1; \qquad
\tilde a_{r+\ell} = - \sum_{t=0}^{r-1} \tilde a_{t+\ell}c_t \quad
\text{for} \quad \ell \ge 0.
\end{equation}
This sequence satisfies the recurrence of \eqref{E:S5}, so that
\begin{equation}\label{E:trick}
\tilde a_j = a_j\quad \text{for}\quad j \le d. 
\end{equation}
Since $|\tilde a_j| \le c\cdot M^j$ for suitable $c,M$, the generating
function 
\begin{equation*} 
\Phi(T) = \sum_{j=0}^\infty \tilde a_j T^j
\end{equation*}
converges in a neighborhood of 0. We have
\begin{equation*}
\begin{gathered}
\left(\sum_{t=0}^r c_{r-t} T^t\right) \Phi(T) =
 \sum_{n=0}^{r-1} \left(\sum_{j=0}^n c_{r-(n-j)}\tilde a_j \right)T^n + 
\sum_{n=r}^\infty \left(\sum_{t=0}^r c_{r-t}\tilde a_{n-t} \right) T^n.
\end{gathered}
\end{equation*} 
It follows from \eqref{E:inf} that the second sum vanishes, and hence
 $\Phi(T)$ is a rational function with denominator
\begin{equation*}
\sum_{t=0}^r c_{r-t} T^t = h(T,1) = \prod_{j=1}^r (1 - \be_j T).
\end{equation*}
By partial fractions, there exist $\la_k \in
\cc$ so that 
\begin{equation}\label{E:reveal}
 \sum_{j=0}^\infty \tilde a_j T^j = \Phi(T) = \sum_{k=1}^r \frac
 {\la_k}{1 - \be_kT} \implies \tilde a_j = \sum_{k=1}^r \la_k\be_k^j.
\end{equation}
A comparison of \eqref{E:reveal} and \eqref{E:trick} with \eqref{E:S1}
shows that
\begin{equation}\label{E:punchline}
\begin{gathered}
f(x,y) = \sum_{j=0}^d \binom dj a_j x^{d-j}y^j = 
\sum_{k=1}^r \la_k \sum_{j=0}^d \binom dj \be_k^j x^{d-j}y^j 
 =\sum_{k=1}^r \la_k (x + \be_k y)^d,
\end{gathered}
\end{equation}
as claimed in \eqref{E:S3}.

If $c_r=0$, then $c_{r-1} \neq 0$, because $h$ has distinct
factors. We may proceed as before, replacing $r$ by $r-1$ and taking
$c_{r-1} = 1$, so that \eqref{E:S2} becomes
\begin{equation}\label{E:S2prime}
h(x,y)
= \sum_{t=0}^{r-1}
c_tx^{r-t}y^t = x \prod_{j=1}^{r-1} (y - \beta_j x).
\end{equation}
Since $c_r=0$, \eqref{E:S4} loses a column and becomes
\begin{equation*}
\begin{pmatrix}
a_0 & a_1 & \cdots & a_{r-1} \\
a_1 & a_2 & \cdots & a_{r}\\
\vdots & \vdots & \ddots & \vdots \\
a_{d-r}& a_{d-r+1} & \cdots & a_{d-1}
\end{pmatrix}
\cdot
\begin{pmatrix}
c_0\\c_1\\ \vdots \\ c_{r-1}
\end{pmatrix}
=\begin{pmatrix}
0\\0\\ \vdots \\ 0
\end{pmatrix}.
\end{equation*}
We argue as before, except that \eqref{E:trick} becomes
\begin{equation}\label{E:trickprime}
\tilde a_j = a_j\quad \text{for}\quad j \le d-1, \qquad a_d = \tilde
a_d + \la_r,
\end{equation}
and \eqref{E:punchline} becomes
\begin{equation}\label{E:punchline2}
\begin{gathered}
f(x,y) = \sum_{j=0}^d \binom dj a_j x^{d-j}y^j = 
\la_r y^d + \sum_{k=1}^{r-1} \la_k \sum_{j=0}^d \binom dj \be_k^j x^{d-j}y^j \\
 =\la_r y^d + \sum_{k=1}^{r-1} \la_k (x + \be_k y)^d.
\end{gathered}
\end{equation}
By \eqref{E:S2prime}, \eqref{E:punchline2} meets the description of
\eqref{E:S3}, completing the proof. 
\end{proof}

The $(d-r+1) \times (r+1)$ Hankel  matrix in \eqref{E:S4} will be
denoted $H_r(f)$. 
If $(f,h)$ satisfy the criterion of this theorem,  we shall say that
$h$ is a {\it Sylvester form   for $f$}. If the only Sylvester forms
of degree $r$ are $\la h$ for $\la \in \cc$, we say that $h$ is the
{\it unique} Sylvester form for $f$. Any multiple of a
Sylvester form that has no repeated factors is also a Sylvester form,
since there is no requirement that $\la_k \neq 0$ in
\eqref{E:S3}. If $f$ has a unique Sylvester form of degree $r$,
then $L_{\cc}(f) = r$ and $L_K(f) \ge r$.

The proof of Theorem 2.1 in \cite{R3} is based on apolarity. If $f$ and
$h$ are given by \eqref{E:S1} and \eqref{E:S2}, and
$h(D) = \prod_{j=1}^{r} (\beta_j 
\frac{\partial}{\partial x} - \alpha_j \frac {\partial}{\partial y})$,
then
\begin{equation*}
h(D)f = \sum_{m=0}^{d-r} \frac{d!}{(d-r-m)!m!}\left( \sum_{i=0}^{d-r}
  a_{i+m}c_i \right) x^{d-r-m}y^m
\end{equation*}
Thus, \eqref{E:S4} is equivalent to $h(D)f=0$. One can then argue that
 each linear factor in $h(D)$ kills a different summand,
and dimension counting takes care of the rest. In particular, if $\deg
h > d$, then $h(D)f=0$ automatically, and this implies that
$L_{\cc}(f) \le d+1$. Theorem 4.2 is a less mysterious explanation of this
fact.  

If $h$ has repeated factors, a condition of interest in
\cite{K1,K2,K3,KR,R3}, then Gundelfinger's Theorem \cite{G}, first
proved in 1886, shows that
a factor $(\be x - \al y)^\ell$ of $h$ corresponds to a summand
$(\al x + \be y)^{d+1-\ell} q(x,y)$ in $f$, where $q$ is an arbitrary
form of degree $\ell - 1$. (Such a summand is unhelpful in the current
context when $\ell \ge 2$.) 
 There is a lengthy history of the connections with the Apolarity
 Theorem in \cite{R3}.

If $d=2s-1$ and $r=s$, then $H_s(f)$ is $s \times (s+1)$ and
has a non-trivial null-vector; for a general $f$, the resulting form
$h$ has distinct factors, and so is a unique Sylvester form. (The
coefficients  of $h$, and its discriminant, are polynomials in the
coefficients of $f$.)  This is how Sylvester proved that a general
binary form of degree $2s-1$ is a sum of $s$ powers of linear forms
and the minimal representation is unique. (If so, $L_K(f) = s$ iff $h$
splits in $K$, but this does not happen in general if $s \ge 2$.) 

If $d=2s$ and $r=s$, then $H_s(f)$ is square; $\det(H_s(f))$ is the
{\it catalecticant} of $f$. (For etymological 
exegeses on ``catalecticant'', see \cite[pp.49-50]{R1} and
\cite[pp.104-105]{Ge}.) 
In general, there exists $\la$ so that the catalecticant of
$f(x,y) - \lambda x^{2s}$ vanishes, and the resulting non-trivial null
vector is generally a Sylvester form (no repeated factors). Thus,
a general binary form of degree $2s$ is a sum of $\lambda
x^{2s}$ plus $s$ powers of linear forms. (It is less clear whether one
should expect $L_K(f) = s+1$ for general $f$.) 

Sylvester's Theorem can be adapted to compute $K$-length when $K
\subsetneq \cc$.  

\begin{corollary}
Given $f \in K[x,y]$, $L_K(f)$ is the minimal degree of a Sylvester
form for $f$ which splits completely over $K$. 
\end{corollary}
\begin{proof}
If \eqref{E:S3} is a minimal representation for $f$ over $K$, where
$\la_k, \al_k, \be_k \in K$, then $h(x,y) \in K[x,y]$ splits over $K$
by \eqref{E:S2}. Conversely, if $h$ is a Sylvester form for $f$ satisfying
\eqref{E:S2} with $\al_k, \be_k \in K$, then \eqref{E:S3} holds for
some $\la_k \in \cc$. This is equivalent to saying that  the linear system
\begin{equation}\label {E:univer}
a_j = \sum_{k=1}^r  \al_k^{d-j} \be_k^j X_k, \quad (0 \le j \le d)
\end{equation}
has a solution $\{X_k = \la_k\}$ over $\cc$. Since $a_j, \al_k^{d-j} \be_k^j
\in K$, it follows that \eqref{E:univer} also has a solution over $K$,
so that $f$ has a $K$-representation of length $r$. 
\end{proof}

We apply these results to the quintic from Example 1.2.

\begin{example}[Continuing Example 1.2]
Note that
\begin{equation*}
\begin{gathered}
\phi(x,y) =    3x^5 -20x^3y^2+10xy^4 =  \binom 50 \cdot 3\ x^5 +
\binom 51 \cdot 0\ x^4y \\ +  \binom 52 \cdot (-2)\ x^3y^2  +  \binom 53
\cdot 0\ x^2y^3 +  \binom 54 \cdot 2\ xy^4 + \binom 55 \cdot 0\ y^5.
\end{gathered}
\end{equation*}
Since
\begin{equation*}
\begin{pmatrix}
3 & 0 & -2 & 0\\
0 & -2 & 0 & 2\\
-2 & 0 & 2 & 0\\
\end{pmatrix}
\cdot
\begin{pmatrix}
c_0\\c_1\\c_2 \\ c_3
\end{pmatrix}
=\begin{pmatrix}
0\\0\\0
\end{pmatrix} 
\iff (c_0,c_1,c_2,c_3) = r(0,1,0,1),
\end{equation*}
$\phi$ has a {\it unique} Sylvester form of degree 3:
$h(x,y) = y(x^2+y^2) = 
y(y - ix)(y+ix)$. Accordingly, there exist $\la_k \in \cc$ so that
\begin{equation*}
\phi(x,y) = \la_1 x^5 + \la_2 (x + i y)^5 + \la_3 (x - i y)^5. 
\end{equation*}
Indeed, $\la_1=\la_2=\la_3 = 1$, as may be  checked. It follows
that $L_K(\phi) = 3$ if and only if $i \in K$. (A representation of
length two would be detected here if some $\la_k = 0$.)

To find representations for $\phi$ of length 4, 
we revisit \eqref{E:S4}:
\begin{equation*}
\begin{gathered}
H_4(\phi) \cdot (c_0,c_1,c_2,c_3,c_4)^t = (0,0)^t  \iff
3c_0-2c_2+2c_4= -2c_1+2c_3 = 0 
\\
\iff (c_0,c_1,c_2,c_3,c_4) = r_1(2,0,3,0,0) + r_2(0,1,0,1,0) +
r_3(0,0,1,0,1), 
\end{gathered}
\end{equation*}
hence $h(x,y) = r_1x^2(2x^2 + 3 y^2) + y(x^2+y^2)(r_2 x + r_3
y)$. Given a field $K$, it is far from obvious whether
there exist $\{r_{\ell}\}$ so that $h$ splits into distinct factors over
$K$. Here are some 
imaginary quadratic fields for which this happens. 

 The choice $(r_1,r_2,r_3) = (1,0,2)$ gives $h(x,y) =
(2x^2+y^2)(x^2+2y^2)$ and
\begin{equation*}
24\phi(x,y) = 4 (x + \sqrt{-2} y)^5 + 4 (x - \sqrt{-2}  y)^5
+  (2x + \sqrt{-2}  y)^5 + (2x - \sqrt{-2}  y)^5.
\end{equation*}
Similarly, $(r_1,r_2,r_3) = (2,0,9)$ and $(2,0,-5)$ give $h(x,y) =
(x^2 + 3 y^2) (4 x^2 + 3 y^2)$ and $ (x^2-y^2)(4x^2+5y^2)$, leading to 
representations for $\phi$ of length 
4 over $\qq(\sqrt{-3})$ and $\qq(\sqrt{-5})$. The simplest such
representation we have found for $\qq(\sqrt{-6})$ uses
$(r_1,r_2,r_3) = (8450,0,-104544)$ and
\begin{equation*}
h(x,y) = (5x + 12y)(5x -12y)(6 \cdot 13^2x^2 + 33^2y^2).
\end{equation*}
It is easy to believe that $L_{\qq(\sqrt{-m})}(\phi) = 4$ for all
squarefree $m \ge 2$, though we have no proof. 
In Example 3.1, we shall show that there is no
choice of $(r_1,r_2,r_3)$ for which $h$ splits into distinct factors over any
subfield of $\rr$.
\end{example}

\section{Sylvester's 1864 Theorem}

Theorem 3.1 was discovered by Sylvester
\cite{S3} in 1864 while proving  Isaac Newton's
conjectural variation on Descartes' Rule of Signs, see
\cite{Ho,Y}. This theorem appeared in P\'olya-Szeg\"o
\cite[Ch.5,Prob.79]{PS}, and has been used  by P\'olya and 
Schoenberg \cite{PSc} and Karlin \cite[p.466]{K}. The (dehomogenized)
version proved in \cite{PS} is: 

\begin{theorem}[Sylvester]
Suppose $0 \neq \la_k$ for all $k$ 
and $\ga_1 < \dots < \ga_r$, $r \ge 2$, are real
numbers such 
that
\begin{equation*}
Q(t) = \sum_{k=1}^r\la_k(t-\ga_k)^d
\end{equation*}
does not vanish identically. Suppose the sequence
$(\la_1,\dots,\la_r,(-1)^d\la_1)$ has $C$ changes of sign and $Q$ has $Z$
zeros, counting multiplicity. Then $Z \le C$. 
\end{theorem}

We shall prove an equivalent version which exploits the
homogeneity of $f$ to avoid discussion of zeros at infinity in the proof.
(The equivalence is discussed in \cite{R4}.)

\begin{theorem}
Suppose $f(x,y)$ is a non-zero real form of degree $d$ with $\tau$
real linear factors (counting multiplicity) and 
\begin{equation}  \label{E:sylrep}
f(x,y) = \sum_{j=1}^r \la_j(\cos \theta_j x + \sin \theta_j y)^d
\end{equation}
where $-\frac {\pi}2 < \theta_1 < \dots < \theta_r \le \frac{\pi}2$,
$r \ge 2$  and $\la_j \neq 0$. Suppose there are $\sigma$ sign changes
in the tuple $(\la_1,\la_2,\dots,\la_r,(-1)^d \la_1)$. Then $\tau \le
\sigma$. In particular, $\tau \le r$.
\end{theorem}

\begin{example}[Examples 1.2 and 2.1 concluded]
Since
\begin{equation*}
 \phi(x,y)  = 3x\bigl( x^2 - \tfrac{10 -
     \sqrt{70}}3y^2\bigr) \bigl( x^2 -  \tfrac{10 + \sqrt{70}}3y^2\bigr) 
\end{equation*}
is a product of five linear factors over $\rr$,  
$L_{\rr}(\phi) \ge 5$. The
representation
\begin{equation*}
6\phi(x,y)= 36x^5  -10 (x + y)^5  - 10 (x -   y)^5 +
(x + 2 y)^5 + (x - 2 y)^5. 
\end{equation*}
over $\qq$ implies that $\mathcal C(\phi) = \{3,4,5\}$.
\end{example}

\begin{proof}[Proof of Theorem 3.2]
We first ``projectivize''  \eqref{E:sylrep}:
\begin{equation}  \label{E:sylrep2}
\begin{gathered} 
2f(x,y) = \sum_{j=1}^r \la_j(\cos \theta_j x
  + \sin \theta_j 
y)^d + \\ 
\sum_{j=1}^r (-1)^d \la_d(\cos (\theta_j +\pi) x + \sin (\theta_j
+ \pi) y)^d
\end{gathered}
\end{equation}
View the sequence $(\la_1,\la_2,\dots,\la_r,(-1)^d \la_1, (-1)^d
\la_2, \dots, (-1)^d \la_r,\la_1)$ cyclically, identifying the first
and last term. There are 2$\sigma$ pairs of consecutive terms with a
negative product. It doesn't matter where one
starts, so if we make any invertible 
change of variables $(x,y) \mapsto (\cos \theta x + \sin \theta y, -\sin
\theta x + \cos \theta y)$ in \eqref{E:sylrep} (which doesn't affect
$\tau$, and which ``dials" the angles by $\theta$), 
and reorder the ``main" angles to $(-\tfrac {\pi}2, \tfrac
{\pi}2]$, the value of $\sigma$ is unchanged. 
 We may therefore assume that neither $x$ nor
$y$ divide $f$, that $x^d$ and $y^d$ are not summands in
\eqref{E:sylrep2} (i.e., 
$\theta_j$ is not a multiple of $\frac{\pi}2$), and that if there
is a sign change in $(\la_1,\la_2,\dots,\la_r)$, then $\theta_u < 0 <
\theta_{u+1}$ implies $\la_u\la_{u+1} < 0$. Under these
hypotheses, we may safely dehomogenize $f$ by setting either $x=1$ or
$y=1$ and avoid zeros at infinity and know that $\tau$ is the number
of zeros of the resulting polynomial. The rest of the proof generally
follows \cite{PS}.  

Let $\bar \sigma$ denote the number of sign changes in
$(\la_1,\la_2,\dots,\la_r)$. We induct on $\bar \sigma$. The base case
is $\bar \sigma = 0$ (and $\la_j > 0$ without loss of
generality). If  $d$ is even, then $\sigma = 0$ and 
\begin{equation*}
f(x,y) =  \sum_{j=1}^r \la_j(\cos \theta_j x + \sin \theta_j y)^d
\end{equation*}
is definite, so $\tau = 0$. If  $d$ is
odd, then $\sigma = 1$. Let $g(t) = f(t,1)$, so that 
\begin{equation*}
g'(t) =  \sum_{j=1}^r d\left(\la_j\cos \theta_j\right)
(\cos \theta_j t + \sin \theta_j)^{d-1}.
\end{equation*}
Since $d-1$ is even, $\cos\theta_j>0$ and $\la_j > 0$,
 $g'$ is definite and $g'\neq 0$. Rolle's Theorem 
implies that $g$ has at most one zero; that is, $\tau \le 1 =
\sigma$. 

Suppose the theorem is valid for $\bar \sigma = m \ge 0$ and suppose
that $\bar\sigma = m+1$ in  \eqref{E:sylrep}. Now let $h(t) =
f(1,t)$. We have
\begin{equation*}
h'(t)  = \sum_{j=1}^r d\left(\la_j\sin \theta_j\right)(\cos \theta_j +
\sin \theta_jt)^{d-1}. 
\end{equation*}
Note that $h'(t) = q(1,t)$, where
\begin{equation*}
q(x,y)  = \sum_{j=1}^r d\left(\la_j\sin \theta_j\right)(\cos \theta_jx +
\sin \theta_jy)^{d-1}. 
\end{equation*}
Since $\bar\sigma \ge 1$, $\theta_u< 0 < \theta_{u+1}$ implies 
that $\la_u\la_{u+1} < 0$, so that the number of sign changes in 
$(d\la_1\sin\theta_1, d\la_2 \sin\theta_2,\dots,d \la_r \sin\theta_r)$
is $m$, as the sign change at the $u$-th consecutive pair has been
removed, and no other possible sign changes are introduced.
 The induction hypothesis implies that $q(x,y)$ has at most
$m$ linear factors, hence $q(1,t) = h'(t)$ has  $\le m$ zeros
(counting multiplicity) and Rolle's Theorem implies that $h$ has
$\le m+1$ zeros, completing the induction.
\end{proof}

\section{Applications to forms of general degree}

We begin with a familiar folklore result:
 the vector space of complex forms $f$ in $n$ variables of degree $d$
is spanned by the set of linear forms
taken to the $d$-th power. It follows from a 1903 theorem of Biermann
(see  \cite[Prop.2.11]{R1} or \cite{R5} for a proof) that a canonical 
set of the ``right'' number of $d$-th powers over $\zz$  forms a basis:
\begin{equation}\label{E:Biermann}
\left\{ (i_1 x_1 + ... + i_n x_n)^d\ :\ 0 \le i_k \in \zz,\ i_1 +
  \cdots + i_n = d \right \}.
\end{equation}
If $f \in K[x_1,\dots,x_n]$, then $f$ is a $K$-linear combination of
these forms and so  $L_K(f) \le \binom {n+d-1}{n-1}$. We show below
(Theorems 4.10, 5.4) that when $n=2$, the bound 
for $L_K(f)$ can be improved from  $d+1$ to $d$, but this is best possible.

The first two simple results are presented explicitly for completeness. 

\begin{theorem}
If $f \in K[x,y]$, then $L_K(f) = 1$ if and only if $L_{\cc}(f) = 1$.
\end{theorem}
\begin{proof}
One direction is immediate from \eqref{E:enlarge}. For the other,
suppose  $f(x,y) 
= (\al x + \be y)^d$ with $\al, \be \in \cc$.
 If $\al = 0$, then $f(x,y) = \be^d y^d$, with
$\be^d \in K$. If $\al \neq 0$, then  $f(x,y) = \al^d(x + (\be/\al)
y)^d$. Since the coefficients of $x^d$ and  $dx^{d-1}y$ in $f$ are $\al^d$
and $\al^{d-1}\be$, it follows that $\al^d$ and $\be/\al =
(\al^{d-1}\be)/\al^d$ are both in $K$.  
\end{proof}

\begin{theorem}
Any set  $\{(\al_j x + \be_j y)^d: 0 \le j \le d\}$ of pairwise distinct
$d$-th powers is linearly independent and spans the binary forms of degree $d$.
\end{theorem}

\begin{proof}
The matrix of this set with respect to the basis $\binom di  x^{d-i}y^i$ is
$[\al_j^{d-i}\be_j^i]$, whose determinant is Vandermonde:
\begin{equation*}
\prod_{0 \le j < k \le d} 
 \begin{vmatrix}
\al_j & \be_j  \\
\al_k&  \be_k
\end{vmatrix}
.
\end{equation*}
This determinant is a product of non-zero terms by hypothesis.
\end{proof}

By considering the difference of two representations of a given form,
we obtain an immediate corollary about different representations of
the same form. Trivial counterexamples, formed by splitting
summands, occur in non-honest representations. 

\begin{corollary}
If $f$ has two different honest representations:
\begin{equation}\label{E:2reps}
f(x,y) = \sum_{i=1}^s \la_i(\al_i x + \be_i y)^d = \sum_{j=1}^t
\mu_j(\ga_j x + \de_j y)^d, 
\end{equation}
then $s+t\ge d+2$. If $s+t=d+2$ in \eqref{E:2reps}, then the combined set of
linear forms, $\{\al_ix+\be_iy, \ga_jx+\de_j y\}$, is
pairwise distinct. 
\end{corollary}

The next result collects some consequences of Corollary 4.3.
\begin{corollary}
Let $E = E_f$. 
\begin{enumerate}
\item If $L_E(f) = r \le \frac d2 + 1$, then $L_{\cc}(f) = r$, so
 $\mathcal C(f) = \{r\}$. 
 \item If, further, $L_E(f) = r \le \frac d2 +
\frac 12$, then $f$ has a unique $\cc$-minimal representation.
\item
If $d=2s-1$ and $H_s(f)$ has full rank and $f$ has
 a unique Sylvester form $h$ of degree $s$, then
$L_K(f) \ge s$, with equality if and only if $h$ splits in $K$. 
\end{enumerate}
\end{corollary} 

\begin{proof}
We take the parts in turn.
\begin{enumerate}
\item
A different representation of $f$ over $\cc$  must have
length $\ge d+2 - r \ge \frac d2 + 1 \ge r$ by Corollary 4.3, and so
$L_{\cc}(f) = r$.
\item If $r \le \frac d2 +\frac 12$, then this
representation has length  $\ge \frac d2 + \frac 32 > r$, and so
cannot be minimal. 
\item If $d=2s-1$ and $r=s$, then the last case applies, so $f$ has a
  unique $\cc$-minimal representation, and by Corollary 2.2, this
  representation can be expressed in $K$ if and only if the Sylvester
  form splits over $K$. 
\end{enumerate}
\end{proof}

We now give some more explicit constructions of forms with multiple
lengths. We first need a lemma about cubics.  
\begin{lemma}
If  $f$ is a cubic given by \eqref{E:S1} and $H_2(f) = \begin{pmatrix}
a_0 & a_1 & a_2\\
a_1 & a_2 & a_3\\ 
\end{pmatrix}$ has rank $\le 1$, then $f$ is a cube.
 \end{lemma}
\begin{proof}
 If $a_0=0$, then $a_1=0$, so $a_2=0$ and $f$ is a cube. If $a_0 \neq
 0$, then $a_2 = a_1^2/a_0$ and $a_3 = a_1a_2/a_0 = a_1^3/a_0^2$ and $f(x,y) =
a_0(x + \frac{a_1}{a_0}y)^3$ is again a cube.
\end{proof}

\begin{theorem}
Suppose $d \ge 3$ and there exist $\al_i, \be_i \in \cc$ so that
\begin{equation}\label{E:two}
f(x,y) = \sum_{i=0}^d \binom di a_i x^{d-i}y^i = (\al_1 x + \be_1 y)^d
+ (\al_2 x + \be_2 y)^d \in K[x,y].
\end{equation}
If \eqref{E:two} is honest and $L_K(f) > 2$, then there exists $u
\in K$ with $\sqrt u \notin K$ so
that $L_{K(\sqrt u)}(f) = 2$. The summands in \eqref{E:two} are
conjugates of each other in $K(\sqrt u)$. 
\end{theorem}

\begin{proof}
 First observe that if
 $\al_2 = 0$, then $\al_2\be_1 \neq \al_1\be_2$ implies that $\al_1
\neq 0$. But then $a_0 = \al_1^d \neq 0$ and $a_1 = \al_1^{d-1}\be_1$
imply that $\al_1^d, \be_1/\al_1 \in K$ as in Theorem 4.1, and so 
\begin{equation*}
f(x,y) - \al_1^d(x + (\be_1/\al_1) y)^d = (\be_2 y)^d = \be_2^d y^d \in K[x,y].
\end{equation*}
This contradicts $L_K(f) > 2$, so $\al_2 \neq 0$; similarly, $\al_1
\neq 0$. Let $\la_i = \al_i^d$ and $\ga_i = \be_i/\al_i$ for
$i=1,2$, so $\la_1\la_2 \neq 0$ and $\ga_1 \neq \ga_2$. We have
\begin{equation*}
f(x,y) = \la_1(x + \ga_1 y)^d + \la_2(x + \ga_2 y)^d \implies a_i =
\la_1\ga_1^i+\la_2\ga_2^i. 
\end{equation*} 
Now let
\begin{equation*}
g(x,y) = \la_1(x + \ga_1 y)^3 + \la_2(x + \ga_2 y)^3 = a_0x^3 +
3a_1x^2y+3a_2xy^2+a_3y^3.  
\end{equation*}
Since $\la_i \neq0$ and \eqref{E:two} is honest, Corollary 3.5 implies
that $L_{\cc}(g) = 2$, so $H_2(g)$ has full rank by Lemma 4.5. It can be
checked directly that 
\begin{equation*}
\begin{pmatrix}
a_0 & a_1 & a_2\\
a_1 & a_2 & a_3\\
\end{pmatrix}
\cdot
\begin{pmatrix}
\ga_1\ga_2\\-(\ga_1+\ga_2)\\1 \\ 
\end{pmatrix}
=\begin{pmatrix}
0\\0
\end{pmatrix}, 
\end{equation*}
and this gives $h(x,y) =(y-\ga_1x)(y-\ga_2x)$ as the unique Sylvester
form for $g$. Since $H_2(g)$ has entries in $K$ and hence has a null
vector in $K$, we must have $h\in K[x,y]$. By hypothesis, $h$ does not
split over $K$; it must do so over $K(\sqrt u)$, where $u = 
(\ga_1-\ga_2)^2   =(\ga_1+\ga_2)^2-4\ga_1\ga_2 \in K$.    Moreover, if
$\sigma$ denotes conjugation with respect to $\sqrt u$, then $\ga_2 =
\sigma(\ga_1)$ and since $\la_1 + \la_2 \in K$, $\la_2 =
\sigma(\la_1)$ as well. Note that $\la_i = \al_i^d$ and $\ga_i =
\be_i/\al_i \in K(\sqrt{u})$, but this is not necessarily true for
$\al_i$ and $\be_i$ themselves.
  \end{proof}

\begin{corollary}
Suppose  $g \in E[x,y]$ does not split over $E$,
but factors into distinct linear factors $g(x,y) = \prod_{j=1}^r (x
+ \al_j y)$ over an extension field $K$ of $E$. If
$d > 2r-1$, then for each $\ell \ge 0$,
$$
f_{\ell}(x,y) = \sum_{j=1}^r \al_j^{\ell} (x+\al_j y)^d \in E[x,y],
$$  
and $L_K(f_{\ell}) = r < d+2-r \le L_{E}(f_{\ell})$.
\end{corollary}

\begin{proof}
The coefficient of $\binom dk x^{d-k}y^k$ in $f_{\ell}$ is $\sum_{j=1}^r
\al_j^{\ell+k}$. Each such power-sum belongs to $E$ by Newton's
Theorem on Symmetric Forms. If 
$\al_s \notin E$ (which must hold for at least one $\al_s \neq 0$), 
then $\al_s^{\ell} (x+\al_s y)^d \notin E[x,y]$.   Apply Corollary 4.3.
\end{proof}

\begin{corollary}
Suppose $K$ is an extension field of $E_f$, $r \le \frac {d+1}2$, and 
\begin{equation*}
f(x,y) =  \sum_{i=1}^r \la_i(\al_i x + \be_i y)^d
\end{equation*}
with $\la_i,\al_i,\be_i \in K$. Then every automorphism of $K$ which
fixes $E_f$ permutes the summands of  the representation of $f$.
\end{corollary}

\begin{proof}
We interpret  $\sigma(\la(\al x + \be y)^d) =  \sigma(\la)(\sigma(\al) x
+ \sigma(\be)  y)^d$. Since  $\sigma(f) = f$, the action of $\sigma$
is to give another representation of $f$. Corollary 4.4(2) implies that
this is the same representation, perhaps reordered.
\end{proof}

This next theorem is undoubtedly ancient, but we
cannot find a suitable reference.
\begin{theorem}
If $f \in K[x,y]$, then  $L_{\cc}(f) \le \deg d$.
\end{theorem}
\begin{proof}
By a change of variables, which does not affect the length,
we may assume that neither $x$ nor $y$
divide $f$, hence $a_0a_d \neq 0$ and $h = a_d x^d - a_0 y^d$ is a
Sylvester form which splits over $\cc$.
\end{proof}
Theorem 4.9 appears as an exercise in  Harris \cite[Ex.11.35]{Ha},
with the (dehomogenized)  maximal length occurring at $x^{d-1}(x+1)$
(see Theorem 5.4). 
 Landsberg and Teitler \cite[Cor. 5.2]{LT} prove that $L_{\cc}(f) \leq
\binom{n+d-1}{n-1} - (n-1)$, which reduces to Theorem 4.9 for
$n=2$. 

The proof given for Theorem 4.9 will not apply to all fields $K$, because
$a_d x^d - a_0 y^d$  usually does not split over $K$. A more careful argument is
required.
\begin{theorem}
If $f \in K[x,y]$, then  $L_K(f) \le \deg d$.
\end{theorem}
\begin{proof}
Write $f$ as in \eqref{E:S1}. If $f$ is identically zero, there is
nothing to prove. 
Otherwise, we may assume that $f(1,0) = a_0\neq 0$ after a change of variables
if necessary.  By Corollary 2.2, it suffices to find 
$h(x,y) =  \sum_{k=0}^d c_kx^{d-k}y^k$ which splits into distinct 
 linear factors over
$K$ and satisfies $\sum_{k=0}^d a_kc_k=0$.

Let $e_0 =1$ and $e_k(t_1,\dots,t_{d-1})$ denote the usual $k$-th
elementary symmetric functions. We make a number of definitions: 
\begin{equation*}
\begin{gathered}
h_0(t_1,\dots,t_{d-1};x,y) := \sum_{k=0}^{d-1}e_k(t_1,\dots,t_{d-1})x^{d-1-k}y^k =
\prod_{j=1}^{d-1} (x + t_j y),\\
 \be(t_1,\dots,t_{d-1}) := -\sum_{k=0}^{d-1}a_k
 e_k(t_1,\dots,t_{d-1}),\\  \al(t_1,\dots,t_{d-1})  :=
\sum_{k=0}^{d-1}a_{k+1} e_k(t_1,\dots,t_{d-1}), \\
\Phi(t_1,\dots,t_{d-1}) := \prod_{j=1}^{d-1} (\al
(t_1,\dots,t_{d-1})t_j - \be(t_1,\dots,t_{d-1})), \\
\Psi(t_1,\dots,t_{d-1}) := \Phi(t_1,\dots,t_{d-1})\prod_{1 \le i < j
  \le d-1}(t_i - t_j). 
\end{gathered}
\end{equation*}
Then $\be(0,\dots,0) = -a_0e_0= -a_0\neq 0$, so $\Phi(0,\dots,0) =
a_0^{d-1} \neq 0$ and $\Phi$ is not the zero polynomial, and thus
neither is $\Psi$. Choose $\ga_j \in K$, $1 \le j \le  d-1$, so that
$\Psi(\ga_1,\dots,\ga_{d-1}) \neq 0$. It follows that the $\ga_j$'s
are distinct, and  $\al \ga_j \neq \be$, where $\al =
\al(\ga_1,\dots,\ga_{d-1})$ and $\be = \be(\ga_1,\dots,\ga_{d-1})$.
Let  $e_k =e_k(\ga_1,\dots,\ga_{d-1})$.  We claim that
\begin{equation*}
\begin{gathered}
h(x,y) = \sum_{i=0}^d c_i x^{d-1}y^i :=
 (\al x + \be y)h_0(\ga_1,\dots,\ga_{d-1};x,y) = (\al x + \be
y)\prod_{j=1}^{d-1}(x + \ga_j y) \\
 =  (\al x + \be y) \sum_{k=0}^{d-1}e_kx^{d-1-k}y^k
= \al e_0 x^d  + \sum_{k=1}^{d-1} (\al e_k + \be e_{k-1})x^{d-k}y^k +
\be e_{d-1}y^d
\end{gathered}
\end{equation*}
is a Sylvester form for $f$. Note that the $\ga_j$'s are distinct and 
$\al\ga_j \neq \be$, $1 \le j \le d-1$, so that $h$ is a product of
distinct linear factors. Finally, 
\begin{equation*}
\begin{gathered}
\sum_{k=0}^d a_kc_k = \al e_0a_0 + \sum_{k=1}^{d-1} (\al e_k + \be
e_{k-1})a_k + \be e_{d-1}a_k = \\ \al \sum_{k=0}^{d-1} e_ka_k + \be
\sum_{k=0}^{d-1} e_ka_{k+1} = \al(-\be) + \be \al = 0. 
\end{gathered}
\end{equation*}
This completes the proof.
\end{proof}

\begin{corollary}
If $f$ is a product of $d$ real linear forms, then $L_{\rr}(f) = d$. 
\end{corollary}
\begin{proof}
Write $f$ as a sum of $L_{\rr}(f) = r \le d$ $d$-th powers and rescale
into  the shape \eqref{E:sylrep}. Taking $\tau = d$ in  Theorem 3.2,
we see that $d \le \sigma \le r$. 
\end{proof}

\begin{conjecture}
If $f \in \rr[x,y]$ is a form of degree $d \ge 3$, then $L_{\rr}(f) = d$ if
and only if $f$ is a product of $d$ linear forms.
\end{conjecture}

We shall see in Theorems 5.2 and 5.3 that this conjecture is true for
$d=3,4$.

\section{Applications to forms of particular degree}

Corollary 4.3 and Theorem 4.10 impose some immediate 
restrictions on the possible cabinets of a form of degree $d$.

\begin{corollary}
Suppose $\deg f =d$. 
\begin{enumerate}
 
\item If $L_{\cc}(f) = r$,  then $\mathcal C(f) \subseteq \{r, d-(r-2),
  d-(r-3), \dots, d\}$. 

\item If $L_{\cc}(f)=2$, then $\mathcal C(f)$ is
  either $\{2\}$ or $\{2,d\}$.

\item If $f$ has $k$ different lengths, then $d \ge 2k-1$.

\item If $f$ is cubic, then  $\mathcal C(f) = \{1\}, \{2\}, \{3\}$ or
$\{2,3\}$.

\item If $f$ is quartic, then  $\mathcal C(f) = \{1\}, \{2\}, \{3\},
  \{4\}, \{2,4\} $ or $\{3,4\}$.
\end{enumerate}
\end{corollary}

We now completely classify $L_K(f)$ when $f$ is a binary cubic.

\begin{theorem}
Suppose $f(x,y) \in E_f[x,y]$ is a cubic form with discriminant
$\Delta$ and suppose  $E_f \subseteq K  \subseteq \cc$.
\begin{enumerate}
\item If $f$ is a cube, then $L_{E_f}(f) = 1$ and $\mathcal C(f) = \{1\}$.
\item If $f$ has a repeated linear factor, but is not a cube, then 
$L_K(f) = 3$ and $\mathcal C(f) = \{3\}$.
\item If $f$ does not have a repeated factor, then $L_K(p) = 2$ if
$\sqrt{-3\Delta} \in K$ and $L_K(p) = 3$ otherwise, so either
$\mathcal C(f) = \{2\}$ or $\mathcal C(f) = \{2,3\}$. 
\end{enumerate}
\end{theorem}
\begin{proof}
The first case follows from Theorem 4.1. In the second case,
after an invertible linear change of variables, we may assume that
$f(x,y) = 3x^2y$, and apply Theorem 2.1 to test for representations of
length 2. But 
\begin{equation}\label{E:x2y}
\begin{pmatrix}
0 & 1 & 0\\
1 & 0 & 0\\ 
\end{pmatrix}
\cdot
\begin{pmatrix}
c_0 \\ c_1 \\ c_2
\end{pmatrix}
=\begin{pmatrix}
0\\0
\end{pmatrix} \implies c_0  = c_1 = 0,
\end{equation}
so  $h$ has repeated factors. Hence
$L_K(x^2y) \ge 3$ and by Theorem 4.10, $L_K(x^2y) = 3$.

Finally, suppose 
\begin{equation*}
f(x,y) = a_0x^3+3a_1x^2y+3a_2xy^2+a_3y^3 = \prod_{j=1}^3 (r_j x + s_j y)
\end{equation*}
does not have repeated factors, so that
\begin{equation*}
0 \neq \Delta(f) = \prod_{j < k} (r_js_k - r_k s_j)^2,
\end{equation*}
and consider the system:
\begin{equation*}
\begin{pmatrix}
a_0 & a_1 & a_2\\
a_1 & a_2 & a_3\\ 
\end{pmatrix}
\cdot
\begin{pmatrix}
c_0 \\ c_1 \\ c_2
\end{pmatrix}
=\begin{pmatrix}
0\\0
\end{pmatrix} 
\end{equation*}
By Lemma 4.5, this system has
rank 2;  the unique Sylvester form is
\begin{equation*}
h(x,y) = (a_1a_3 - a_2^2)x^2 + (a_1a_2 - a_0a_3) xy + (a_0a_2-a_1^2)y^2,
\end{equation*}
which happens to be the Hessian of $f$.
Since $h \in E_f[x,y] \subseteq K[x,y]$, it splits over $K$ if and only
if its  discriminant is a square in $K$. 
A computation shows that 
\begin{equation*}
(a_1a_2 - a_0a_3)^2 - 4(a_1a_3 - a_2^2)(a_0a_2-a_1^2) = -
\frac{\Delta(f)}{27} = -\frac{3\Delta(f)}{9^2}.
\end{equation*}
Thus, $L_K(f) = 2$ if and only if $\sqrt{-3\Delta(f)} \in K$.
If $h$ does not split over $F$, then $L_F(f) = 3$ by Theorem 4.10.
\end{proof}
In particular, $x^3$, $x^3+y^3$, $x^2 y$ and $(x+iy)^3 + (x-iy)^3$ have the
cabinets enumerated in Corollary 5.1(4).
If $f$ has three distinct real linear factors, then $\Delta(f) >
0$, so $\sqrt{-3\Delta(f)} \notin \rr$ and $L_{\rr}(f) = 3$. If $f$ is
real and has one real 
and two conjugate complex linear factors, then $\Delta(f) < 0$, so $L_{\rr}(f) =
2$. Counting repeated roots, we see that if $f$ is a real cubic, and
not a cube, then $L_{\rr}(f) = 3$ if and only if it has three real
factors, thus proving Conjecture 4.12 when $d=3$.

\begin{example}
We find all representations of $3x^2y$ of length 3. Note that
\begin{equation*}
H_3(f) \cdot (c_0,c_1,c_2,c_3)^t = (0) \iff c_1 = 0 \iff
h(x,y) = c_0x^3 + c_2 xy^2 + c_3y^3. 
\end{equation*}
If $c_0 = 0$, then $y^2 \ | \ h$, which is to be avoided,
so we  scale and assume $c_0=1$.
We can parameterize the Sylvester forms  $h(x,y)=(x - a y)(x - b y)(x
+ (a+b)y)$ with 
$a,b,-(a+b)$ distinct. This leads to an easily checked  general formula
\begin{equation}\label{E:3reps}
\begin{gathered}
3(a-b)(a+2b)(2a+b)x^2 y = \\ (a + 2 b) (a x + y)^3 - (2 a + b) (b x +
y)^3 + (a -  b) (-(a + b) x + y)^3.
\end{gathered}
\end{equation}
It is not hard to find analogues of \eqref{E:3reps} for $d > 3$; we leave this
to the reader.
\end{example}

\begin{theorem}
If $f$ is a real quartic form, then $L_{\rr}(f) = 4$ if
and only if $f$ is a product of four linear factors. 
\end{theorem}
\begin{proof}
Factor $\pm f$ as a product of $k$ positive definite quadratic forms
and $4-2k$ linear forms. If $k=0$, then Corollary 4.11 implies that
$L_{\rr}(f) = 4$. We must show that if $k=1$ or $k=2$, then $f$ has a
representation over $\rr$ as a sum of $\le 3$ fourth powers. 

If $k=2$, then $f$ is positive definite and by \cite[Thm.6]{PR}, after
an invertible linear change of variables, $f(x,y) = x^4 + 6\la
x^2y^2+y^4$, with $6\la \in (-2,2]$. (This is also proved in
\cite{R5}.) If $r \neq 1$, then 
\begin{equation}\label{E:q1}
\begin{gathered}
 (rx + y)^4 + (x+ry)^4 -
(r^3+r)(x+y)^4  \\
= (r-1)^2(r^2+r+1) \left(x^4 - \left(\tfrac{6r}{r^2+r+1}\right)
x^2y^2 + y^4\right). 
\end{gathered}
\end{equation}
Let $\phi(r) =  - \frac{6r}{r^2+r+1}$. Then $\phi(-2+\sqrt 3) = 2$ and
$\phi(1) = -2$, and since $\phi$ is continuous, it maps $[-2+\sqrt
3,1)$ onto $(-2,2]$, and \eqref{E:q1} shows that $L_{\rr}(f) \le 3$.

If $k=1$, there are two cases, depending on whether the linear
factors are distinct. Suppose that after a linear
change, $f(x,y) = x^2h(x,y)$, where $h$ is positive definite, and so
for some $\la > 0$ and linear $\ell$, $h(x,y) = \la x^2 +
\ell^2$. After another linear change, 
\begin{equation}\label{E:q2}
f(x,y) = x^2(2x^2 + 12y^2) = (x+y)^4 + (x-y)^4 - 2y^4,
\end{equation}
and \eqref{E:q2} shows that $L_{\rr}(f) \le 3$.

If the linear factors are distinct, then after a linear change,
\begin{equation*}
f(x,y) = xy(ax^2 + 2b x y + cy^2)
\end{equation*}
where $a>0, c>0, b^2 < ac$. After a scaling,
$f(x,y) = xy(x^2 + d x y + y^2)$, $|d| 
< 2$, and by taking $\pm f(x,\pm y)$, we may assume $d \in [0,2)$.
 If $r \neq 1$, then
\begin{equation}\label{E:q3}
\begin{gathered}
(r^4+1)(x+y)^4 - (rx + y)^4 -
(x+ry)^4  \\ 
= 4(r-1)^2(r^2+r+1) \left( x^3y + \left(\tfrac{3(1+r)^2}{2(r^2+r+1)}\right) x^2y^2 + xy^3\right).
\end{gathered}
\end{equation}
Let $\psi(r) = \frac{3(1+r)^2}{2(r^2+r+1)}$. Since $\psi(-1) = 0$,
$\psi(1) = 2$ and  $\psi$ is continuous, it maps $[-1,1)$ onto
$[0,2)$, and \eqref{E:q3} shows that $L_{\rr}(f) \le 3$. 
 \end{proof}

The next result must be ancient; $L_{\cc}(x^{d-1}y)=d$ seems well
known, but we have not found a suitable reference for the converse.
Landsberg and Teitler \cite[Cor.4.5]{LT} show that $L_{\cc}(x^ay^b) =
\max(a+1,b+1)$ if $a,b \ge 1$. 

\begin{theorem}
If $d \ge 3$, then $L_{\cc}(f) = d$ if and only if there are two distinct
linear forms $\ell$ and $\ell'$ so that $f = \ell^{d-1}\ell'$.
\end{theorem} 
\begin{proof}
If $f = \ell^{d-1}\ell'$, then after an invertible linear change, we
may assume that $f(x,y) = d x^{d-1}y$. If $L_{\cc}(dx^{d-1}y) \le d-1$,
then $f$ would have a Sylvester form of degree $d-1$. But then, as in
\eqref{E:x2y}, \eqref{E:S4} becomes
\begin{equation*}
\begin{pmatrix}
0 & 1 & \cdots & 0 \\
1 & 0 & \cdots & 0 
\end{pmatrix}
\cdot
\begin{pmatrix}
c_0\\c_1\\ \vdots \\ c_{d-1}
\end{pmatrix}
=\begin{pmatrix}
0\\0
\end{pmatrix} \implies c_0 = c_1 = 0,
\end{equation*}
so $h$ does not have distinct factors. Thus,  $L_{\cc}(dx^{d-1}y) =d$.

Conversely, suppose $L_{\cc}(f) = d$. Factor $f = \prod \ell_j^{m_j}$
as a product of pairwise distinct linear forms, with $\sum m_j = d$,
$m_1 \ge m_2 \dots \ge m_s\ge 1$, and $s > 1$ (otherwise, $L_{\cc}(f) =
1$.) Make an invertible linear change taking $(\ell_1,\ell_2)
\mapsto (x,y)$, and call the new form $g$; $L_{\cc}(g) = d$ as well.
If $g(x,y) = \sum_{\ell=0}^d \binom d{\ell}
b_{\ell} x^{d-\ell}y^{\ell}$, then $b_0 = b_d=0$. 
 By  hypothesis, there does not exist 
a Sylvester form of degree $d-1$ for $g$. Consider
\begin{equation*}
\begin{pmatrix}
0 & b_1 & \cdots & b_{d-2} & b_{d-1} \\
b_1 & b_2 & \cdots & b_{d-1} & 0
\end{pmatrix}
\cdot
\begin{pmatrix}
c_0\\c_1\\ \vdots \\ c_{d-1}
\end{pmatrix}
=\begin{pmatrix}
0\\0
\end{pmatrix}.
\end{equation*}   
If $m_1 \ge m_2 \ge 2$, then $x^2,y^2 \ |\ g(x,y)$ and $b_1 = b_{d-1} =
0$ and $x^{d-1} - y^{d-1}$ 
is a Sylvester form of degree $d-1$ for $f$. Thus $m_2 = 1$ and so
$y^2$ does not divide $g$ and $b_1 \neq 0$. 
Let $q(t) = \sum_{i=0}^{d-2} b_{i+1}t^i$ (note the absence of binomial
coefficients!) and suppose $q(t_0) = 0$. Since $q(0) = b_1$, $t_0 \neq
0$. We have 
\begin{equation*}
\begin{pmatrix}
0 & b_1 & \cdots & b_{d-2} & b_{d-1} \\
b_1 & b_2 & \cdots & b_{d-1} & 0
\end{pmatrix}
\cdot
\begin{pmatrix}
1\\t_0\\ \vdots \\ t_0^{d-1}
\end{pmatrix}
=\begin{pmatrix}
t_0q(t_0)\\q(t_0)
\end{pmatrix}
=\begin{pmatrix}
0\\0
\end{pmatrix}.
\end{equation*}  
Since
\begin{equation*}
h(x,y) = \sum_{i=0}^{d-1} t_0^i x^{d-1-i}y^i = \frac{x^d - t_0^d y^d}{x - t_0
  y} = \prod_{k=1}^{d-1} (x - \ze_{d-1}^k t_0y)
\end{equation*} 
has distinct linear factors, it is a Sylvester form for $g$, and
$L_{\cc}(g) \le d-1$. This contradiction implies that $q$ has no
zeros, and by the Fundamental Theorem of Algebra, $q(t) = b_1$
must be a constant. It follows that $g(x,y) = db_1x^{d-1}y$, as promised.
\end{proof}

 By Corollaries 4.4 and 5.1, instances of the first five cabinets in
 Corollary 5.1(5) are: $x^4$, $x^4+y^4$, $x^4 + y^4 + (x+y)^4$, $x^3y$
 and $(x+iy)^4 + (x-iy)^4$. It will follow from the next results that 
$\mathcal C((x^2+y^2)^2) = \{3,4\}$.

\begin{theorem}
If $d=2k$ and $f(x,y) = \binom {2k}k x^ky^k$, then $L_{\cc}(f) = k+1$.
The minimal $\cc$-representations of $f$
are given by 
\begin{equation}\label{E:kent}
(k+1)\binom {2k}kx^ky^k =
\sum_{j=0}^k  (\ze_{2k+2}^j w x + \ze_{2k+2}^{-j} w^{-1}y)^{2k},
\qquad 0 \neq w \in \cc. 
\end{equation}
\end{theorem}
\begin{proof}
We first evaluate the right-hand side of \eqref{E:kent} by expanding the
$2k$-th power:
\begin{equation}\label{E:surprise}
\begin{gathered}
\sum_{j=0}^k  (\ze_{2k+2}^j w x + \ze_{2k+2}^{-j} w^{-1}y)^{2k}
= \sum_{j=0}^k \sum_{t=0}^{2k} \binom {2k}t \ze_{2k+2}^{j(2k-t)-jt}
w^{(2k-t)-t} x^{2k-t}y^t \\
=  \sum_{t=0}^{2k} \binom {2k}t w^{2k-2t} x^{2k-t}y^t\left( \sum_{j=0}^k
\ze_{k+1}^{j(k-t)}\right).
\end{gathered}
\end{equation}
But $ \sum_{j=0}^{m-1}\ze_m^{rj} = 0$ unless $m \ | \ r$, in
which case it equals $m$. Since the only multiple of $k+1$ in the set
$\{k-t: 0 \le t \le 2k\}$ occurs for $t=k$, \eqref{E:surprise} reduces to
the left-hand side of \eqref{E:kent}. We now show that
these are {\it all} the minimal $\cc$-representations of $f$.

Since $H_k(x^ky^k)$ has
1's on the NE-SW diagonal, it is non-singular, so
$L_{\cc}(x^ky^k) > k$, and $L_{\cc}(x^ky^k) = k+1$ by
\eqref{E:kent}. By Corollary 4.3, any minimal
$\cc$-representation {\it not} given by \eqref{E:kent} can only use powers
of forms which are distinct from any $w x + w^{-1}y$. If $ab = c^2 \neq 
0$, then $a x+ by$ is a multiple of $\frac ac x + \frac ca y$. This
leaves only $x^{2k}$ and $y^{2k}$, and there is  no 
linear combination of these giving $x^ky^k$. 
\end{proof}

The representations in \eqref{E:kent} arise because the
null-vectors of $H_{k+1}(x^ky^k)$ can only be $(c_0,0,\dots,0,c_{k+1})^t$ and
$c_0x^{k+1} + c_{k+1}y^{k+1}$ is a Sylvester form when $c_0c_{k+1} \neq 0$.

\begin{corollary}
For $k \ge 1$, $L_{\cc}((x^2+y^2)^k) = k+1$, and $L_K((x^2+y^2)^k) =
k+1$ iff $\tan \frac {\pi}{k+1} \in K$. The $\cc$-minimal
representations of $(x^2+y^2)^k$ are given by 
\begin{equation}\label{E:kent2}
\binom dk (x^2+y^2)^k = \frac 1{k+1} \sum_{j=0}^k
\left(\cos(\tfrac{j\pi}{k+1}+ \theta) x +  \sin(\tfrac{j\pi}{k+1}+
  \theta) y\right)^{d}, \quad \theta \in \cc.
\end{equation}
\end{corollary}
\begin{proof}
The invertible  map $(x,y) \mapsto (x-iy, x+iy)$ takes $x^ky^k$ into
 $(x^2+y^2)^k$. Setting $0 \neq w = e^{i\theta}$ in
\eqref{E:kent} gives \eqref{E:kent2} after the usual reduction.
If $\tan\al \neq 0$, then
\begin{equation*}
(\cos \al\ x + \sin \al\ y)^{2r} = \cos^{2r} \al \cdot
(x + \tan \al\  y)^r = (1 + \tan^2\al)^{-r} (x + \tan \al\  y)^r. 
\end{equation*}
Thus, $(\cos \al\ x + \sin \al\ y)^{2r} \in K[x,y]$ iff $ \cos \al
= 0$ or $\tan \alpha \in K$. It follows that $L_K((x^2+y^2)^k) = k+1$
if and only if there exists $\theta \in\cc$ so that for each $0 \le j
\le k$, either  
$\cos(\tfrac{j\pi}{k+1}+ \theta) = 0$ or $\tan(\tfrac{j\pi}{k+1}+
\theta) \in K$. 
Since $\tan \alpha, \tan \beta \in K$ imply $\tan (\al-\be) \in K$ and
$k \ge 1$, we see that
 \eqref{E:kent2} is a
representation over $K$ if and only if $\tan \frac{\pi}{k+1} \in K$.
\end{proof}
 
In particular, since $\tan \frac{\pi}3 = \sqrt 3 \notin \qq$,
$L_{\qq}((x^2+y^2)^2 > 3$ and so must equal 4. Thus,
$\mathcal C((x^2+y^2)^2) = \{3,4\}$, as promised. Since $\tan \frac
{\pi}m$ is irrational for $m \ge 5$ (see e.g. \cite[Cor.3.12]{Ni}), it
follows that $L_{\qq}((x^2+y^2)^k) = k+1$ only for $k=1,3$.

It is worth remarking that $x^ky^k$ is a highly singular complex form,
as is $(x^2+y^2)^k$. However, as a {\it real} form,  $(x^2+y^2)^k$ is
in some sense at the center of the cone $Q_{2,2k}$. 
For  real $\theta$,  the formula in \eqref{E:kent2} 
goes back at least to Friedman \cite{F} in 1957.
It was shown in  \cite{R1} that all minimal {\it real} representations
of $(x^2+y^2)^k$ have this shape. 
There is an equivalence between representations of $(x^2+y^2)^k$ as a real
sum of $2k$-th powers and quadrature formulas on the circle -- see
\cite{R1}. In this sense, \eqref{E:kent2} can be traced back to Mehler
\cite{M} in 1864. Taking  $k=7, \theta = 0$ and
$\rho:= \tan \frac{\pi}8 =\sqrt 2 - 1$ in \eqref{E:kent2} gives 
\begin{equation*}
\begin{gathered}\tfrac{429}{256}(x^2 + y^2)^7 =
x^{14} + y^{14} + \tfrac 1{128}\left( (x+y)^{14} + (x -y)^{14}
\right) \\
+ \left(\tfrac{2+\sqrt 2}4 \right)^7 \left( (x + \rho y)^{14} +  (x - \rho
  y)^{14} + (\rho x + y)^{14} +  (\rho x - y)^{14}\right).
\end{gathered}
\end{equation*}

A real representation \eqref{E:basic} of $(\sum x_i^2)^k$ 
(with positive real coefficients $\la_j$) is called a Hilbert Identity; 
Hilbert  \cite{H, E2} used such representations with
rational coefficients to solve Waring's problem. Hilbert Identities are
deeply involved with quadrature problems on $S^{n-1}$, the
Delsarte-Goethals-Seidel theory of spherical designs in combinatorics
and for embedding questions in Banach spaces \cite[Ch.8,9]{R1}, as well
as for explicit computations in Hilbert's 17th problem \cite{R2}.
It can be shown that any such representation requires at least
$\binom{n+k-1}{n-1}$ summands, and this bound
also applies if negative coefficients $\la_j$ are allowed. It is not
known whether allowing negative coefficients can reduce to the total
number of summands. When $(\sum x_i^2)^k$ is a sum of exactly
$\binom{n+k-1}{n-1}$  $2k$-th powers, the
coordinates of minimal representations can be used to  
produce tight spherical designs. Such representations exist 
when $n=2$, $2k=2$, $(n,2k) = (3,4)$, $(n,2k) =
(u^2-2,4)$ ($u=3,5$), $(n,2k) = (3v^2-4,6)$ ($v=2,3)$, $(n,2k) =
(24,10)$. It has been proved that they do not exist otherwise, unless
possibly  $(n,2k) =(u^2-2,4)$ for some odd
integer $u \ge 7$ or $(n, 2k) =  (3v^2-4,6)$ for some integer $v \ge
4$. These questions have been largely open for thirty years. 
It is also not known whether there
exist $(k,n)$ so that $L_{\rr} ((\sum x_i^2)^k)) > L_{\cc}
((\sum x_i^2)^k)$, although this cannot happen for $n=2$.
For that matter, it is not known whether there exists any $f \in
Q_{n,d}$ so that $L_{\rr}(f) > L_{\cc}(f)$. 

We conclude this section with a related question: if $f_{\la}(x,y) =
x^4 + 6\la x^2y^2 + y^4$ for $\la \in \qq$, what is
$L_{\qq}(f_{\la})$? If $\la \le -\frac 13$, then $f_{\la}$ has four
real factors, so $L_{\qq}(f_{\la})= 4$.  
Since $\det H_2(f_\la) = \la -
\la^3$,  $L_{\cc}(f_{\la}) = 2$ for $\la = 0, 1, -1$. The formula 
\begin{equation*}
(x^4 + 6\la x^2 y^2 + y^4) = \tfrac{\la}2 \left((x+y)^4 +
  (x-y)^4\right) + (1-\la)(x^4 + y^4). 
\end{equation*}
shows that $L_{\qq}(f_0) = L_{\qq}(f_1) = 2$; $2f_{-1}(x,y) = (x+iy)^4
+(x-iy)^4$ has $\qq$-length 4.

\begin{theorem}
Suppose $\la = \frac ab \in \qq, \la^3 \neq \la$. Then $L_{\qq}(x^4 +
6\la x^2 y^2 + y^4) = 3$  
if and only if there exist  integers $(m,n)\neq(0,0)$ so that
\begin{equation}\label{E:quartic}
\Gamma(a,b,m,n) = 4a^3 b\ m^4 +(b^4 - 6a^2b^2 - 3a^4) m^2n^2 + 4a^3 b\ n^4
\end{equation}
is a non-zero square. 
\end{theorem}
\begin{proof}
By Corollary 2.2, such a representation occurs if and only if there is
a cubic $h(x,y) = \sum_{i=0}^3 c_ix^{3-i}y^i$ which splits over $\qq$
and satisfies
\begin{equation}\label{E:last}
c_0 + \la c_2 = \la c_1 + c_3 = 0.
\end{equation}
Assume that $h(x,y) = (mx+ny)g(x,y)$, $(m,n) \neq (0,0)$ with
$m,n \in \zz$.
 If $g(x,y) = rx^2 + s xy + ty^2$, then $c_0 = mr,
c_1 = ms + nr, c_2 = mt + ns, c_3 = nt$ and \eqref{E:last} becomes
\begin{equation}\label{E:last2}
 \begin{pmatrix}
m & \la n & \la m\\
\la n & \la m & n\\ 
\end{pmatrix}
\cdot
\begin{pmatrix}
r \\ s \\ t
\end{pmatrix}
=\begin{pmatrix}
0\\0
\end{pmatrix} 
\end{equation}
If $m=0$, then the general solution to \eqref{E:last2}
is $(r,s,t) = (r,0,-\la r)$ and
$rx^2 - \la r y^2$ splits over $\qq$ into distinct factors iff $\la$ is
a non-zero square; that is,
iff $ab$ is a square, and similarly if  $n=0$. Otherwise,  the system
has full rank since $\la^2 \neq 1$ and any solution is a multiple of 
\begin{equation}\label{E:Qhilb}
rx^2 + s xy + ty^2= (\la n^2 - \la^2 m^2) x^2 + (\la^2-1)mn xy + (\la
m^2 - \la^2 n^2)y^2. 
\end{equation}
The quadratic in \eqref{E:Qhilb} splits over $\qq$ into distinct
factors iff its discriminant  
\begin{equation}
\begin{gathered}
4\la^3 m^4 + (1 - 6\la^2 - 3\la^4)m^2n^2 + 4\la^3 n^4=  b^{-4}\Gamma(a,b,m,n)
\end{gathered}
\end{equation}
is a non-zero square in $\qq$.
\end{proof}

In particular, we have the following identities: $\Gamma(u^2,v^2,v,u)
= (u^5v - u v^5)^2$ and $\Gamma(uv, u^2-uv+v^2,1,1)
= (u-v)^6(u+v)^2$,
hence $L_{\qq}(f_{\la})=3$ for $\la = \tau^2$ and $\la =
\frac{\tau}{\tau^2 - \tau+1}$, where $\tau = \frac uv \in \qq$,  $\tau
\neq\pm  1$.  These show that $L_{\qq}(f_{\la})=3$ for a dense set
of rationals in $[-\frac 13, \infty)$. These families do not exhaust the
possibilities. If $\la = \frac{38}3$, so $f_{\la}(x,y) = x^4 +
76x^2y^2+y^4$, then $\la$ is expressible neither as $\tau^2$ nor
$\frac{\tau}{\tau^2 - \tau+1}$ for $\tau \in \qq$, but
$\Gamma(38,3,2,19) = 276906^2$. 

We mention two negative cases: if $\lambda = \frac 13$,
 $\Gamma(1,3,m,n) = 12(m^2+n^2)^2$, which is
never a square, giving another proof that 
$L_{\qq}((x^2+y^2)^2) = 4$. 
If $\la = \frac 12$, then 
\begin{equation*}
\Gamma(1,2,m,n) = 8m^4 -11m^2n^2+8n^4 = \tfrac {27}4(m^2-n^2)^2 + \tfrac
54(m^2 + n^2)^2, 
\end{equation*}
hence if $L_{\qq}(x^4 + 3x^2y^2 + y^4) = 3$, then there is a solution
to the Diophantine equation $27 X^2 + 5 Y^2 = Z^2$. A simple descent
shows that this has no non-zero solutions: working mod 5, we
see that $2X^2 = Z^2$; since 2 is not a quadratic residue mod 5, it follows
that $5\ |\ X, Z$, and these imply that $5\ |\ Y$ as well.

Solutions of the Diophantine equation $A m^4 + B m^2 n^2 + C n^4 =
r^2$ were first 
studied by Euler; see \cite{Di}[pp.634-639] and \cite{Mor}[pp.16-29]
for more on this topic. This equation has not yet been completely
solved; see \cite{Br, Co}. We hope to return to the analysis of
\eqref{E:quartic} in a future publication.

\section{Open Questions}

Conecture 4.12 seems plausible, but as the degree increases, the
canonical forms become increasingly involved. Are there other fields
besides $\cc$ (and possibly $\rr$) for which there is a simple
description of $\{f : L_K(f) = \deg f\}$?

Which cabinets are possible? Are there other restrictions beyond
Corollary 5.1(1)?  How many different lengths are possible? If
$|\mathcal C(f) | \ge 4$, then  $d \ge 7$. 

Can $f$ have more than one, but a finite number, of $K$-minimal
 representations, where $K$ is not necessarily equal to $E_f$? Theorem
 5.7 might be a way to find such examples.

Length is generic over $\cc$, but not over $\rr$. For $d=2r$, the
  $\rr$-length of a real form is always $2r$ in a small neighborhood of 
  $\prod_{j=1}^d (x - jy)$, but the $\rr$-length is always $r+1$ in a
  small neighborhood of $(x^2+y^2)^r$, by \cite{R1}. Which
combinations of degrees and lengths have interior? Does the parity of
$d$ matter? 


\end{document}